\documentclass{article}
\usepackage[margin=1.25in]{geometry}
\usepackage{amsmath}
\usepackage{amssymb}
\usepackage{amsthm}
\usepackage[all]{xy}
\usepackage{hyperref}
\usepackage{makeidx}
\usepackage{mathtools}
\usepackage{dsfont}

\providecommand{\Epi}{\textnormal{Epi}}
\providecommand{\Rep}{\textnormal{Rep}}
\providecommand{\Bcom}{B_{\textnormal{com}}}
\providecommand{\Hom}{\textnormal{Hom}}
\providecommand{\ann}{\textnormal{ann}}
\providecommand{\ab}{\textnormal{ab}}
\providecommand{\Sq}{\textnormal{Sq}}
\providecommand{\Z}{\mathbb{Z}}
\providecommand{\C}{\mathbb{C}}
\providecommand{\R}{\mathbb{R}}
\providecommand{\Q}{\mathbb{Q}}
\providecommand{\RP}{\mathbb{RP}}

\providecommand{\F}{\mathbb{F}}
\providecommand{\bigast}{\scalebox{1.5}{\raisebox{-0.2ex}{$\ast$}}}
\providecommand{\colim}{\mathop{\textnormal{colim}}}

\newtheorem{lemma}{Lemma}
\newtheorem{prop}{Proposition}
\newtheorem{theorem}{Theorem}
\newtheorem{corl}{Corollary}
\newtheorem*{introthm}{Theorem \ref{\introthmref}}
\newtheorem*{introcorl}{Corollary \ref{\introthmref}}
\newtheorem*{introprop}{Proposition \ref{\introthmref}}

\theoremstyle{definition}
\newtheorem{defi}{Definition}
\newtheorem{exam}{Example}
\newtheorem{remark}{Remark}

\title{Nilpotent $n$-tuples in $SU(2)$\footnotetext{AMS Subject Classification 2010: 22E99 (primary), 55R35 (secondary)}}
\author{Omar Antol\'{i}n--Camarena and Bernardo Villarreal\thanks{The second author was supported by a CONACyT fellowship.}}
\date{\today}

\begin{document}

\maketitle

\begin{abstract}
We describe the connected components of the space $\Hom(\Gamma,SU(2))$ of homomorphisms for a discrete nilpotent group $\Gamma$. The connected components arising from homomorphisms with non-abelian image turn out to be homeomorphic to $\RP^3$. We give explicit calculations when $\Gamma$ is a finitely generated \emph{free nilpotent} group. In the second part of the paper we study the filtration $\Bcom SU(2)=B(2,SU(2))\subset\cdots \subset B(q,SU(2))\subset\cdots$ of the classifying space $BSU(2)$ (introduced by Adem, Cohen and Torres-Giese), showing that for every $q\geq2$, the inclusions induce a homology isomorphism with coefficients over a ring in which 2 is invertible.  Most of the computations are done for $SO(3)$ and $U(2)$ as well.
\end{abstract}

\tableofcontents

\section{Introduction}
Given a discrete group $\Gamma$ and a topological group $G$, we topologize the set of homomorphisms $\Hom(\Gamma,G)$ as a subspace of $G^\Gamma$, where the latter space is given the product topology.
In the most commonly studied case, $\Gamma$ is finitely generated. If $\Gamma$ has $k$ generators, $\Hom(\Gamma,G)$ is in one to one correspondence with the subset of the product $G^k$ consisting of $k$-tuples satisfying the relations that hold among the generators, and the topology we gave $\Hom(\Gamma,G)$ agrees with its topology as a subspace of $G^k$. In this paper we give a full description of the spaces of homomorphisms of a discrete nilpotent group $\Gamma$, to the Lie group $SU(2)$. We do this by showing that all non-abelian nilpotent subgroups of $SU(2)$ are conjugate to one of the generalized quaternions $Q_{2^q}$ of order $2^q$.

\newcommand{\introthmref}{main}
\begin{introthm}\label{main}
Let $\Gamma$ be a discrete nilpotent group, and for every $q\geq 3$ let $\Epi(\Gamma, Q_{2^q})$ denote the set of surjective homomorphisms $\Gamma\to Q_{2^q}$. Then there is a homeomorphism
\[\Hom(\Gamma,SU(2))\cong \Hom(\Gamma^\ab,SU(2))\sqcup\bigsqcup_{q\geq 3}\bigsqcup_{\mathcal{O}}PU(2)\]
where $\mathcal{O}$ runs through the orbits $\Epi(\Gamma,Q_{2^q})/N_{SU(2)}(Q_{2^q})$ induced by the action of conjugation of the normalizer in $SU(2)$.
\end{introthm}

To have a full description of $\Hom(\Gamma, SU(2))$ one also needs a description of $\Hom(\Gamma^\ab, SU(2))$, which is afforded by the following proposition for a general discrete abelian group.

\renewcommand{\introthmref}{hps}
\begin{introprop}
Let $A$ be a discrete abelian group. There is a pushout square which is also a homotopy pushout square,
\[\xymatrix{SU(2)/N(T)\times \Hom(A,\{\pm I\})\ar[r]\ar[d]&SU(2)/T\times_{\Z/2} \Hom(A,T)\ar[d]\\
\Hom(A,\{\pm I\})\ar[r]&\Hom(A,SU(2))}\] 
where $T$ is a maximal torus of $SU(2)$, $N(T)$ its normalizer; the horizontal arrows are induced by the inclusions $\{\pm I\}\hookrightarrow T$, $\{\pm I\}\hookrightarrow SU(2)$; the vertical arrows are induced by conjugation, and the action of $\Z/2$ on $SU(2)/T\times \Hom(A,T)$ is diagonal, given by the antipodal action on $SU(2)/T\cong S^2$ and by complex conjugation on $\Hom(A,T)$.

\end{introprop}

As an application of Theorem \ref{main} we compute the number of connected components of the space $\Hom(F_n/\Gamma^{q+1}_n,G)$ where $F_n/\Gamma^{q+1}_n$ is the free nilpotent group on $n$ generators of nilpotency class $q$, and $G=SU(2), SO(3)$ or $U(2)$. For example:

\renewcommand{\introthmref}{Theorem1}
\begin{introcorl}
Let $q\geq 2$ and $n\geq 1$. Then
\[\Hom(F_n/\Gamma^{q+1}_n,SU(2))\cong \Hom(\Z^n,SU(2))\sqcup \bigsqcup_{C(n,q+1)}PU(2)\]
where 
\[C(n,q+1)=\frac{2^{n-2}(2^n-1)(2^{n-1}-1)}{3}+(2^n-1)(2^{(q-2)(n-1)}-1)2^{2n-3}.\]
\end{introcorl}

In section \ref{sec:so3-u2}, we cover the cases $G=SO(3)$ and $U(2)$. Section \ref{sec:um} contains a description of the representation space $\Rep(F_n/\Gamma^3_n,U(m))$ (the orbit space associated to the conjugation action of $U(m)$ over $\Hom(F_n/\Gamma^3_n,U(m))$) that expresses it as a union of almost commuting tuples of block matrix subgroups of $U(m)$ that can be identified with the direct product $U(m_1)\times\cdots\times U(m_k)$ where $m_1+\cdots+m_k=m$.

The spaces $\Hom(F_n/\Gamma^{q}_n,G)$ assemble into a simplicial space whose geometric realization is denoted by $B(q,G)$. These spaces yield a filtration \[B(2,G) \subseteq B(3,G) \subseteq \cdots \subseteq BG\] originally introduced by Adem, Cohen and Torres-Giese \cite{Ad5}. Using our work in section 1, and \cite[Theorem 4.6]{Ad5} applied to $Q_{2^{n}}$ we see that
\[B(r,Q_{2^{n}})\simeq B\left(\bigast_{\Z/{2^{r-1}}}^{2^{n-r}}Q_{2^{r}} *_{\Z/{2^{r-1}}} \Z/{2^{n-1}}\right)\]
where $*_A$ denotes the amalgamated product along $A$. The above formula allowed us to do some computations in cohomology for $r=2$. In this case, the space $B(2,G)$ is denoted by $\Bcom G$.\\

\renewcommand{\introthmref}{cr2gq}
\begin{introprop}\leavevmode
\begin{enumerate}
\item There is an isomorphism of graded rings $H^*(\Bcom Q_8;\F_2)\cong \F_2[y_1,y_2,y_3,z]/(y_iy_j,y_1^2+y_2^2+y_3^2,i\ne j)$ where $y_i$ has degree 1 and $z$ degree 2. 

\item Let $n\geq 4$. Then the $\F_2$-cohomology ring of $\Bcom Q_{2^{n}}$ has a presentation given by
\[H^*(\Bcom Q_{2^{n}};\F_2)\cong \F_2[x_1,x_2,y_1,\ldots,y_{2^{n-2}},z]/(x_1^2,x_ky_i,y_iy_j,i\ne j ,x_2+\sum_{i=1}^{2^{n-2}}y_i^2),\]
where $x_1,y_j$ have degree 1 and $x_2,z$ have degree 2.
\end{enumerate}
\end{introprop}

Increasing the degree of nilpotency class makes cohomology calculations considerably harder as we illustrate for the case of $B(3,Q_{16})$ (part of our calculation uses {\sc Singular} \cite{Singular}).

In the last section, we describe the homotopy type of the spaces $B(q,SU(2))$ ---which classify principal $SU(2)$-bundles with transitional nilpotency class less than $q$, as defined by Adem, G\'omez, Lind and Tillman \cite{Ad3}. We show there is a homotopy pushout square
\[\xymatrix{
PU(2)\times_{N_{q+1}} B({q},Q_{2^{q+1}})\ar[d]\ar[r]&B({q},SU(2))\ar[d]\\
PU(2)\times_{N_{q+1}}BQ_{2^{q+1}}\ar[r]&B(q+1,SU(2)),
}\]
where $N_{q+1}$ is the normalizer of the dihedral group $D_{2^{q}}$ in $PU(2)$. 

\renewcommand{\introthmref}{main2}
\begin{introthm}\label{main2}
The inclusions in the filtration
\[\Bcom SU(2)\subset B(3,SU(2))\subset\cdots\subset B(q,SU(2))\subset\cdots\]
are homology isomorphisms with coefficients over a ring $R$ where $2$ is invertible. Moreover, the cohomology ring of each term of the filtration is given by $R[c_2,y_1]/(y_1^2)$, where $\deg(c_2)=\deg(y_1)=4$, and $c_2$ is the pullback of the second Chern class $c_2\in H^4(BSU(2);R)$ under the inclusion $B(q,SU(2))\to BSU(2)$.
\end{introthm}

For $q\geq 2$, the 2-fold covering map $SU(2)\to SO(3)$ induces $B\Z/2$-principal bundles $B(q+1,SU(2))\to B(q,SO(3))$, and we conclude:

\renewcommand{\introthmref}{corl3}
\begin{introcorl}
Let $q\geq 2$. Then the maps $B(q,SU(2))\to B(q,SO(3))$ induced by the double cover $SU(2)\to SO(3)$ are homology isomorphisms with coefficients over a ring $R$ in which $2$ is invertible. 
\end{introcorl}

\begin{remark}
It is shown in \cite[Theorem 1]{Bergeron} that given a finitely generated nilpotent group $\Gamma$, a complex reductive linear group $G$ and $K\subset G$ a maximal compact subgroup, the inclusion induces a strong deformation retract $\Hom(\Gamma,G)\simeq \Hom(\Gamma, K)$. Applying this to $G=SL(2,\C),SL(3,\R)$ or $GL(2,\C)$ and $K=SU(2),SO(3)$ or $U(2)$ respectively, we get descriptions of $\Hom(\Gamma,G)$, and $B(q,G)$ for all $q\geq 2$, as well.
\end{remark}

\section*{Acknowledgments}
We would like to thank Jos\'e Manuel G\'omez for helpful suggestions about $U(2)$ and for motivating us to look into $SU(m)$ and $U(m)$ for $m>2$.

\section{Nilpotent subgroups of $SU(2)$}

In this section we record some purely algebraic facts about $SU(2)$, chief among them a description of all of its non-abelian nilpotent subgroups. Let $T\subset SU(2)$ denote the maximal torus
\[T=\left\{\left(\begin{matrix}
\lambda&0\\
0&\overline{\lambda}
\end{matrix}\right)
\;|\;\lambda\in \C \text{ and }|\lambda|=1\right\}\]
and let $w=\begin{pmatrix}
  0&-1\\
  1&0\end{pmatrix}$. Consider the commutator $[x,y]=xyx^{-1}y^{-1}$. A straightforward calculation shows:

\begin{lemma}\label{commutator-formulas}
Let $x,y\in T$. Then $wx=\overline{x}w$ and
\begin{enumerate}
\item $[x,y]=1$;
\item $[x,wy]=x^2$;
\item $[wx,y]=\overline{y}^2$;
\item $[wx,wy]=\overline{x}^2y^2$.
\end{enumerate}
\end{lemma}

\begin{defi}
Let $Q$ be a group, define inductively $\Gamma^1(Q)=Q$; $\Gamma^{q+1}(Q)=[\Gamma^q(Q),Q]$. This is called the descending central series of $Q$
\[\cdots\subset\Gamma^{q}(Q)\subset\cdots\subset\Gamma^2(Q)\subset \Gamma^1(Q) = Q.\] 
A group $Q$ is \emph{nilpotent} if $\Gamma^{q+1}(Q)=1$ for some $q$. The least such integer $q$ is called the \emph{nilpotency class} of $Q$.
\end{defi}

Let $\xi_n = \begin{pmatrix} e^{2\pi i/n} & 0 \\ 0 & e^{-2\pi
    i/n} \end{pmatrix} \in T$ be an $n$-th root of unity and let
$\mu_n$ stand for the subgroup generated by $\xi_n$. The general
quaternions \[Q_{2^{q+1}}:=\mu_{2^{q}}\cup w\mu_{2^{q}}\] have
nilpotency class $q$ for any $q\geq 1$. Indeed, for every $n\geq1$ the commutators $[\xi_{2^n},w]=\xi^2_{2^n}=\xi_{2^{n-1}}$, then one can show inductively that for $r\geq1$ the $(r+1)^{\textnormal{st}}$ stage in the descending central series is $\Gamma^{r+1}(Q_{2^{q+1}})=\mu_{2^{q-r}}$, and hence the descending central series of $Q_{2^{q+1}}$
\[1\subset\mu_2\subset\cdots\subset\mu_{2^{q-2}}\subset\mu_{2^{q-1}}\subset Q_{2^{q+1}}\]
has $q$ non-trivial terms.

\begin{lemma}\label{Lemma4}
Let $H\subset T\cup wT$ be a subgroup. Suppose $\Gamma^r(H)=\mu_{2^n}$ for some $n>0$. Then if $r> 2$, $\Gamma^{r-1}(H)=\mu_{2^{n+1}}$, and for $r=2$ there exists $t\in T$ such that $tHt^{-1}=Q_{2^{n+2}}$. 
\end{lemma}
\begin{proof}
  We prove the case $r=2$. Suppose $\Gamma^2(H)=[H,H]=\mu_{2^{n}}$. Since the commutator generates $\mu_{2^n}$, there must exist $z_0,wx_0\in H$, with $x_0\in T$, such that $[z_0,wx_0]=\xi_{2^n}^k$ for some odd $k$ (otherwise we would have $[H,H] \subseteq \mu_{2^{n-1}}$). If $z\in H\cap T$, then $[z,wx_0]=z^2$ is a power of $\xi_{2^n}$, that is, $z$ is a power of $\pm\xi_{2^{n+1}}$ and $z\in \mu_{2^{n+1}}$. And for $wy\in H\cap wT$, where $y\in T$, the commutator $[wy,wx_0]=\overline{y}^2{x_0}^2=\xi_{2^n}^p$ for some $p>0$, by Lemma \ref{commutator-formulas}, hence $y=\pm\xi_{2^{n+1}}^px_0$. Therefore $H\subseteq \mu_{2^{n+1}}\cup w\mu_{2^{n+1}}x_0$.

  To get the equality, we turn back to the element $z_0\in H$ chosen above. There are two possibilities: 
\begin{itemize}
\item If $z_0 \in T$, then $z_0^2=[z_0,wx_0]=\xi_{2^{n}}^k$, so $z_0=\pm\xi_{2^{n+1}}^k$.
\item If $z_0 \in wT$, say $z_0=wz'_0$, then $\overline{z'_0}^2 x_0^2 = [wz'_0,wx_0] = \xi_{2^{n}}^k$ so that $z_0^\prime=\pm\xi_{2^{n+1}}^kx_0$.
\end{itemize}

In either case $z_0$ and $wx_0z_0$ generate $\mu_{2^{n+1}}\cup w\mu_{2^{n+1}}x_0$. Conjugating any element $w\xi_{2^{n+1}}^px_0$ by $t=\sqrt{x_0}\in T$ we obtain
\[\sqrt{x_0}w\xi_{2^{n+1}}^px_0\overline{\sqrt{x_0}}= w\overline{\sqrt{x_0}}\xi_{2^{n+1}}^px_0\overline{\sqrt{x_0}}= w\xi_{2^{n+1}}^p.\]
This is independent of the choice of the branch cut for $\sqrt{x_0}$. Hence $\Gamma^1(t H t^{-1})=t H t^{-1}\subseteq \mu_{2^{n+1}}\cup w\mu_{2^{n+1}}$.

For $r>2$ it's easier: the same arguments without the anti-diagonal matrix cases prove the result, since $\Gamma^{r-1}(H)\subset T$.
\end{proof}

\begin{lemma}\label{simult-diag}
Let $X,Y\in SU(2)$. Then 
\begin{enumerate}
\item Suppose $X\ne \pm I$. If $[X,Y]=I$ and $g\in SU(2)$ diagonalizes $X$, then it also diagonalizes $Y$.
\item If $[X,Y]=-I$ and $g\in SU(2)$ diagonalizes $X$, then $gXg^{-1}=\pm \xi_4$ and $gYg^{-1}\in wT$.
\end{enumerate}
\end{lemma}
\begin{proof}
  \begin{enumerate}
  \item When $X \neq \pm I$ its eigenvalues are distinct conjugate
    numbers. So $gYg^{-1}$ commutes with a non-scalar diagonal
    matrix and must also be diagonal.
  \item Let $g$ be a matrix that conjugates $X$ to a diagonal matrix
    with eigenvalues $\lambda,\overline{\lambda}$. Let
    $X^\prime=gXg^{-1}$ and $Y^\prime=gYg^{-1}$. Choose a non-zero
    vector $v\in E_\lambda$, the eigenspace of $X$ associated to
    $\lambda$. Since $[X^\prime,Y^\prime]=-I$ we get
    \[X^\prime Y^\prime v=-Y^\prime X^\prime v=-\lambda Y^\prime v.\]
    Hence $-\lambda$ is an eigenvalue of $X$ so that
    $-\lambda=\overline{\lambda}$, which implies $\lambda=\pm i\in\C$.
    Therefore $X^\prime=\pm\xi_4$. Now, $v\in E_i$ and
    $Y^\prime v\in E_{-i}$ tells us that $Y$ is conjugated by $g$ into
    an anti-diagonal matrix.
\end{enumerate}
\end{proof}

\begin{prop}\label{Prop1}
Let $H\subset SU(2)$ be a nilpotent subgroup. Then either $H$ is
abelian or there exists an $r \ge 2$ and an element $g\in SU(2)$ such that 
\[g Hg^{-1}=Q_{2^{r+1}}.\]
\end{prop}

In particular, any non-abelian nilpotent subgroup of $SU(2)$ is automatically finite.

\begin{proof}
Suppose $H$ is non-abelian. Then, there exists a unique $r>1$ such that $\Gamma^{r+1}(H)=I$ and $\Gamma^r(H)\ne I$. This implies that $\Gamma^r(H)$ sits inside the center of $H$. Non-abelian subgroups of $SU(2)$ have center contained in $\{\pm I\}$ (for a proof of this statement see \cite[Example 2.22, p.~16]{Villarreal}) and therefore $\Gamma^r(H)=[H,\Gamma^{r-1}(H)]=\mu_2$. By hypothesis, there exists a non-central element $X$ in $\Gamma^{r-1}(H)$. Then for every $h\in H$, the commutator $[X,h]$ is inside $\{\pm I\}$. Let $g\in SU(2)$ be an element that diagonalizes $X$. By Lemma \ref{simult-diag}, for every $h\in H$, $ghg^{-1}\in T\cup wT$ and thus $H$ can be conjugated to a subgroup of $T\cup wT$. Now, $\Gamma^r(gHg^{-1})=\mu_2$ and applying  Lemma \ref{Lemma4} inductively we get that there exists $t\in T$, such that $tgH(tg)^{-1}=Q_{2^{r+1}}$. 
\end{proof}

In the next section we will need to know a little bit more about the generalized quaternion groups: their normalizers in $SU(2)$ and a list of their subgroups.

\begin{lemma}\label{Lemma6}
The normalizer of $Q_{2^n}$ in $SU(2)$ is:
\begin{enumerate}
\item The Binary Octahedral group which is generated by
  $\left\{\xi_8,w,\frac{1}{2}\left(\begin{matrix}1+i&-1+i\\1+i&1-i
      \end{matrix}\right)\right\}$ for $n=3$;
\item $Q_{2^{n+1}}$ for $n>3$.
\end{enumerate}
\end{lemma}
\begin{proof}
  For this it is more convenient to think of $SU(2)$ as the group of unit quaternions, which enables one to think geometrically in terms of rotations in the space of purely imaginary quaternions, a space we shall identify with $\mathbb{R}^3$. If $q = \cos(\theta) + \sin(\theta) \hat{q}$ where $\hat{q}$ is a unit length imaginary quaternion, then $v \mapsto qvq^{-1}$ maps imaginary quaternions to imaginary quaternions and is a rotation around $\hat{q}$ of angle $2\theta$. An arbitrary quaternion can be written as $a + v$ where $a$ is the real part and $v$ is purely imaginary. We have $q(a+v)q^{-1} = a + qvq^{-1}$, so that conjugation by $q$ preserves the real part and rotates the imaginary part. This rule that associates a rotation of $\mathbb{R}^3$ to each unit quaternion is the double cover homomorphism $SU(2) \to SO(3)$.

  The normalizer of $Q_{2^n}$ consists of those unit quaternions whose corresponding rotations are the orientation preserving symmetries of the set of imaginary parts of the elements of $Q_{2^n}$. As quaternions, the elements of $Q_{2^n}$ are $\cos(2 \pi m/2^{n-1}) + i \sin(2 \pi m/2^{n-1})$ and $j\cos(2 \pi m/2^{n-1}) + k\sin(2 \pi m/2^{n-1})$ for $m=0,1,\ldots,2^{n-1}-1$. The imaginary parts are then $2^{n-2}$ points on the $i$-axis together with the vertices of a regular $2^{n-1}$-gon in the $jk$-plane. For $n>3$, any orientation preserving isometry of $\mathbb{R}^3$ preserving this set of points must separately preserve the $2^{n-1}$-gon and the points on the $i$-axis, and thus must be a rotation of $\mathbb{R}^3$ whose restriction to the plane of the $2^{n-1}$-gon is a symmetry of the $2^{n-1}$-gon. Each of the $2^n$ elements of the dihedral group of symmetries of the $2^{n-1}$-gon extends to a unique rotation of $\mathbb{R}^3$ (the rotations in the dihedral group extend to rotations around the $i$-axis, while the reflections extend to rotations of angle $\pi$ around the axis of the reflection), and each such rotation of $\mathbb{R}^3$ has two quaternions inducing it; these $2^{n+1}$ quaternions are easily seen to be the elements of $Q_{2^{n+1}}$.

  Now, if $n=3$, there is extra symmetry because the $2^{n-1}$-gon in the $jk$-plane is just a square, and the $2^{n-2} = 2$ points on the $i$-axis together with that square form a regular octahedron. Therefore the normalizer of $Q_{2^3}$ is the preimage in $SU(2)$ of the group of orientation preserving symmetries of the octahedron in $SO(3)$; this preimage is known as the Binary Octahedral group. The generator $\frac{1}{2} \left( \begin{matrix} 1+i & -1+i \\ 1+i & 1-i \end{matrix} \right)$ listed in the statement of the proposition (which is the only generator not in $Q_{2^4}$), corresponds to a rotation of $2\pi/3$ around a line connecting the centers of two opposite faces of the octahedron. This extra symmetry does \emph{not} separately preserve the $i$-axis and the square $\{\pm j, \pm k\}$.
\end{proof}


 

\begin{lemma}\label{Lemma7}
Let $n\geq 2$. Then
\begin{enumerate}
\item The abelian subgroups of $Q_{2^{n+1}}$ are all subgroups of $\mu_{2^{n}}$ or $\{\pm I, \pm wx\}$ where $x$ is an element of $\mu_{2^{n}}$.

\item The non-abelian subgroups of $Q_{2^{n+1}}$ are of the form $\mu_{2^{r}}\cup wx\mu_{2^{r}}$ where $x=(\xi_{2^{n}})^{p}$, for some $2\leq r\leq n$ and some $0 \le p < 2^{n-r}$·
\end{enumerate}
\end{lemma}

\begin{proof}\leavevmode
  \begin{enumerate}
  \item By Lemma \ref{commutator-formulas}, any subgroup containing an element of $\mu_{2^{n}}-\{\pm I\}$ and one of $w\mu_{2^{n}}$ has non-trivial commutator.

  \item By the proof of Lemma \ref{Lemma4} any non-abelian nilpotent subgroup of $T\cup wT$ has the form $\mu_{2^r}\cup wx\mu_{2^r}$ for some $x\in T$ and $r\geq 2$. Therefore the non-abelian subgroups of $Q_{2^{n}}$ are $\mu_{2^{r}}\cup wx\mu_{2^{r}}$ where $x$ ranges over representatives for the cosets of $\mu_{2^{r}}$ in $\mu_{2^{n}}$.
  \end{enumerate}
\end{proof}

\section{Spaces of  homomorphisms from nilpotent groups to $SU(2)$}

Now that we know what all the non-abelian nilpotent subgroups of $SU(2)$ are, we can describe spaces of homomorphisms from discrete nilpotent groups into $SU(2)$.

\begin{theorem}\label{main}
Let $\Gamma$ be a discrete nilpotent group, and for every $q\geq 3$ let $\Epi(\Gamma, Q_{2^q})$ denote the set of surjective homomorphisms $\Gamma\to Q_{2^q}$. Then there is a homeomorphism
\[\Hom(\Gamma,SU(2))\cong \Hom(\Gamma^\ab,SU(2))\sqcup\bigsqcup_{q\geq 3}\bigsqcup_{\mathcal{O}}PU(2)\]
where $\mathcal{O}$ runs through the orbits $\Epi(\Gamma,Q_{2^q})/N_{SU(2)}(Q_{2^q})$ induced by the action of conjugation of the normalizer in $SU(2)$.
\end{theorem}
\begin{proof}
Let $\rho\in \Hom(\Gamma,SU(2))-\Hom(\Gamma^{ab},SU(2))$. By Proposition \ref{Prop1}, $\rho(\Gamma)$ is conjugate to $Q_{2^q}$ for some $q\geq 3$. Then, $\rho$ factors through $\Gamma\xrightarrow{\varphi}Q_{2^q}\hookrightarrow SU(2)\xrightarrow{c_g}SU(2)$, where $\varphi$ is a surjective homomorphism and $c_g$ is conjugation by some $g\in SU(2)$. Thus, we have a surjective map
\[SU(2)\times \bigsqcup_{q\geq 3}\Epi(\Gamma, Q_{2^q})\to \Hom(\Gamma,SU(2))-\Hom(\Gamma^{ab},SU(2))\] 
given by $(g,\varphi)\mapsto g\varphi g^{-1}$. Now two different elements $(g,\varphi)$ and $(h,\psi)$ have the same image if and only if $\varphi,\psi\in \Epi(\Gamma, Q_{2^q})$ for the same $q\geq 3$, and $g^{-1}h\psi h^{-1}g=\varphi$; therefore $g^{-1}h\in N_{SU(2)}(Q_{2^q})$. Thus, after modding out by the action of $N_{SU(2)}(Q_{2^q})$ by conjugation over $\Epi(\Gamma,Q_{2^q})$, we can assume that $\varphi=\psi$, which implies $g^{-1}h\in Z(\langle Q_{2^q} \cup \{g^{-1}h\}\rangle)=\{\pm I\}$ (see the proof of Lemma \ref{simult-diag}). It follows that $h=-g$. Modding out by the left translation action of $\{\pm I\}$ on $SU(2)$ gives the desired homeomorphism.
\end{proof}

Our main family of examples will be the finitely generated free nilpotent groups. By definition, these are the groups $F_n/\Gamma^k_n$, where $\Gamma^k_n := \Gamma^k(F_n)$.

\begin{corl}\label{Theorem1}
Let $q\geq 2$ and $n\geq 1$. Then
\[\Hom(F_n/\Gamma^{q+1}_n,SU(2))\cong \Hom(\Z^n,SU(2))\sqcup \bigsqcup_{C(n,q+1)}PU(2)\]
where 
\[C(n,q+1)=\frac{2^{n-2}(2^n-1)(2^{n-1}-1)}{3}+(2^n-1)(2^{(q-2)(n-1)}-1)2^{2n-3}.\]
\end{corl}
\begin{proof}

In order to have an epimorphism $F_n/\Gamma^{q+1}_n\to Q_{2^r}$ we must have $r\leq q$. And when $r \le q$, we need only choose $n$-tuples in $Q_{2^r}$ that generate the whole group. Let us identify $\Epi(F_n/\Gamma^q_n,Q_{2^r})$ with the subset of $(Q_{2^r})^n$ consisting of $n$-tuples whose entries generate $Q_{2^r}$. The action of $N_{SU(2)}(Q_{2^{r}})$ on $\Epi(F_n/\Gamma^q_n,Q_{2^r})$ is free once we take quotient by the center of the group, which acts trivially. By Theorem \ref{main}, the number of connected components homeomorphic to $PU(2)$ (which arise from non-abelian representations) is
\[\sum_{r=3}^{q+1}\frac{|\Epi(F_n/\Gamma^q_n,Q_{2^r})|}{|N_{SU(2)}(Q_{2^r})|/2}.\]
By Lemma \ref{Lemma6} we know the order of $N_{SU(2)}(Q_{2^r})$. It remains to count the number of elements in $\Epi(F_n/\Gamma^q_n,Q_{2^r})$, for which we use an inclusion--exclusion argument: a tuple generates $Q_{2^r}$ if and only if its elements don't all come from a single maximal proper subgroup of $Q_{2^r}$. So if $M_1, M_2, \ldots, M_m$ are the maximal subgroups of $Q_{2^r}$, we have
\[|\Epi(F_n/\Gamma^q_n,Q_{2^r})| = |Q_{2^r}|^n- \sum_i |M_i|^n + \sum_{i,j} |M_i \cap M_j|^n - \sum_{i,j,k} |M_i \cap M_j \cap M_k|^n + \cdots\]
For $Q_8$, there are three maximal subgroups, all isomorphic to $\mathbb{Z}/4$. The intersection of any two of them is the center of $Q_8$, ${\pm I}$. So we get
\[|\Epi(F_n/\Gamma^q_n,Q_{8})| = 8^n - 3 \cdot 4^n + 3 \cdot 2^n - 2^n = 2^{n+1}(2^n-1)(2^{n-1}-1).\]
For $r \ge 4$, it follows from Lemma \ref{Lemma7} that things are pretty much the same: there are just three maximal subgroups of $Q_{2^r}$, namely, $\mu_{2^{r-1}}$, $\mu_{2^{r-2}} \cup w\mu_{2^{r-2}}$ and $\mu_{2^{r-2}} \cup w \xi_{2^{r-1}} \mu_{2^{r-2}}$. And the intersection of any pair of them is $\mu_{2^{r-2}}$, so we get
\[|\Epi(F_n/\Gamma^q_n,Q_{2^r})| = (2^r)^n - 3 \cdot (2^{r-1})^n + 3 \cdot (2^{r-2})^n - (2^{r-2})^n = 2^{(r-2)n+1}(2^n-1)(2^{n-1}-1),\]
and a quick calculation verifies the formula claimed.
\end{proof}

\begin{remark}
As noted in \cite{Villarreal}, the space $\Hom(F_n/\Gamma^3_n,SU(2))$ is the same as the space of almost commuting tuples $B_n(SU(2),\{\pm I\})$. In \cite{Ad4} they describe this space, and it agrees with our computation for $q=2$. 
\end{remark}

We finish this section with a geometric application of Theorem \ref{main}.

\begin{corl}
For any $q\geq 2$, consider the 3-manifolds $M_q=SU(2)/Q_{2^q}$. Then there are $3+2^{q-2}$ isomorphism classes of flat principal $SL(2,\C)$-bundles over $M_q$.
\end{corl}

\begin{proof}
  For any Lie group $G$, the number of isomorphism classes of flat $G$-bundles is equal to the number of points in the space $\Rep(\pi_1(M_q),G) = \Hom(\pi_1(M_q), G)/G$ (where the orbits are taken modulo the conjugation action of $G$). Our results will let us easily count this number for $G = SU(2)$, and then we shall argue that the answer is the same for $G=SL(2,\C)$.

  \medskip

  According to Theorem \ref{main}, there are two types of components in $\Hom(\pi_1(M_q), SU(2))$: the components of $\Hom(\pi_1(M_q)^\ab, SU(2))$, and several components homeomorphic to $PU(2)$; each of the latter is a conjugacy class of some epimorphism $\pi_1(M_q) \to Q_{2^r}$. Since $\pi_1(M_q)^\ab \cong \Z/2 \times \Z/2$ and the only element of order $2$ in $SU(2)$ is $-I$ (and the same is true in $SL(2, \C)$), $\Hom(\pi_1(M_q)^\ab,SU(2))$ consists of $4$ points. The other components contribute a point each to $\Rep(\pi_1(M_q), SU(2))$, so it is a finite discrete space of cardinality given by:
  \[|\Hom(Q_{2^q}^\ab,SU(2))|+\sum_{r=3}^{q}|\Epi(Q_{2^q},Q_{2^r})/N_{SU(2)}(Q_{2^r})|.\]
  
Let's show that for $r\geq 4$ we have $|\Epi(Q_{2^q},Q_{2^r})/N_{SU(2)}(Q_{2^r})|=2^{2r-3}/2^r=2^{r-3}$, and that for $r=3$, $|\Epi(Q_{2^q},Q_{8})/N_{SU(2)}(Q_{8})|=1$ ---which gives the stated count, but for $G=SU(2)$.

Let $r \ge 4$. A homomorphism $Q_{2^q} \to Q_{2^r}$ is determined by where it sends the generators $w \mapsto x$ and $\xi_{2^{q-1}} \mapsto y$, and we must count the number of choices for which the homomorphism is well-defined and surjective. First of all, since $x^4 = I$ we cannot have $x \in \mu_{2^{r-1}}$. Indeed, in that case we would need to have $y \in w \mu_{2^{r-1}}$ (to avoid having $\langle x, y \rangle \subseteq \mu_{2^{r-1}}$), but then $\langle x, y \rangle \cap \mu_{2^{r-1}} = \langle x, y^2 \rangle \subseteq \mu_4 \subsetneq \mu_{2^{r-1}}$. So $x = wx'$ for $x' \in \mu_{2^{r-1}}$ and there are $2^{r-1}$ choices for $x'$.

We can either have $y \in \mu_{2^{r-1}}$ or $y \in w \mu_{2^{r-1}}$. In the first case, $y$ alone is responsible for generating $\mu_{2^{r-1}}$, which leaves $2^{r-2}$ possibilities for $y$. In the second case, say $y=wy'$ with $y' \in \mu_{2^{r-1}}$, we have $\langle x, y \rangle \cap \mu_{2^{r-1}} = \langle xy \rangle = \langle - \overline{x'} y' \rangle$, which again gives $2^{r-2}$ possibilities for $y'$. (It remains to check all these possibilities are realized which is straightforward using the presentation $Q_{2^r} = \langle \xi_{2^{r-1}}, w \mid w^4 = 1, \xi_{2^{r-1}}^{2^{r-1}} = 1, w\xi_{2^{r-1}} = \xi_{2^{r-1}}^{-1}w \rangle$.) All together we have $2^{2r-3}$ epimorphisms. Again the normalizer acts freely after modding out the center which acts trivially, so $|\Epi(Q_{2^q},Q_{2^r})/N_{SU(2)}(Q_{2^r})|=2^{2r-3}/2^r=2^{r-3}$ as claimed.

The case $r=3$ is straightforward. One can verify that there are 24 pairs $(x,y)$ that generate $Q_8$, all of these pairs define homomorphisms from $Q_{2^q}$, and $N_{SU(2)}(Q_8)/\{\pm I\}$ acts simply transitively on the 24 pairs.

\medskip

Now it only remains to show that the count is the same for $G=SL(2,\C)$. By \cite[Theorem 1]{Bergeron}, the inclusion $SU(2) \to SL(2,\C)$ induces a homotopy equivalence $\Hom(\pi_1(M_q),SU(2)) \simeq \Hom(\pi_1(M_q),SL(2,\C))$, so in particular these spaces have the same number of connected components. Since both $SU(2)$ and $SL(2,\C)$ are connected, this implies that the corresponding $\Rep$ spaces also have the same number of connected components, so we need only show that for $G=SL(2,\C)$ each component of $\Rep(\pi_1(M_q), G)$ is a single point, as was the case for $G = SU(2)$. Consider the map $\Rep(\pi_1(M_q), SU(2)) \to \Rep(\pi_1(M_q), SL(2,\C))$ induced by the inclusion. Since any homomorphism from the finite group $\pi_1(M_q)$ to $SL(2,\C)$ can be conjugated into $SU(2)$ (by the classical averaging argument producing a $\pi_1(M_q)$-invariant Hermitian metric), this map is surjective. Thus the components of $\Rep(\pi_1(M_q), SL(2,\C))$ must also be points, as required.
\end{proof}

\subsection{Spaces of homomorphisms from abelian groups to $SU(2)$}

For a complete description of spaces of homomorphisms from nilpotent groups to $SU(2)$,  it only remains to describe the space $\Hom(A,SU(2))$ for a discrete abelian group $A$.

\begin{prop}\label{hps}
Let $A$ be a discrete abelian group. There is a pushout square which is also a homotopy pushout square,
\[\xymatrix{SU(2)/N(T)\times \Hom(A,\{\pm I\})\ar[r]\ar[d]&SU(2)/T\times_{\Z/2} \Hom(A,T)\ar[d]\\
\Hom(A,\{\pm I\})\ar[r]&\Hom(A,SU(2))}\] 
where $T$ is a maximal torus of $SU(2)$, $N(T)$ its normalizer; the horizontal arrows are induced by the inclusions $\{\pm I\}\hookrightarrow T$, $\{\pm I\}\hookrightarrow SU(2)$; the vertical arrows are induced by conjugation, and the action of $\Z/2$ on $SU(2)/T\times \Hom(A,T)$ is diagonal, given by the antipodal action on $SU(2)/T\cong S^2$ and by complex conjugation on $\Hom(A,T)$.
\end{prop}

\begin{proof}
  It is enough to show the square is a pushout since $SU(2)/T \times \Hom(A, T)$ has a $\Z/2$-CW-complex structure for which $SU(2)/T \times \Hom(A, \{\pm I\})$ is the sub-complex of fixed points, so the top map is a cofibration.

  The right vertical map is the one induced by the map $G/T \times \Hom(A,T) \to \Hom(A, SU(2))$ given by $(gT, \phi) \mapsto (a \mapsto g \phi(a) g^{-1})$. This descends to the quotient by the described $\Z/2$-action because that action has a more conceptual description, as can easily be checked: $\Z/2$ arises here as the Weyl group, $N(T)/T$; and it acts on $\Hom(A,T)$ by conjugation on the codomain of the homomorphisms, namely $nT \cdot \phi = a \mapsto n \phi(a) n^{-1}$ for $nT \in N(T)/T$, $\phi \in \Hom(A, T)$, and it acts on $G/T$ by $nT \cdot gT = gn^{-1}T$ for $nT \in N(T)/T$, $gT \in G/T$.
  
  In $SU(2)$ every abelian subgroup is conjugate to a subgroup of $T$, and therefore the right vertical map is surjective. Let's figure out how non-injective it is. If $(gT,\phi)$ and $(g'T, \phi')$ have the same image in $\Hom(A, SU(2))$, then $h \phi'(a) h^{-1} = \phi(a)$ for all $a \in A$ where $h = g^{-1} g'$. This implies that $\phi(A) \subseteq T \cap h T h^{-1}$. There are two cases:

  \begin{itemize}
  \item If $T = h T h^{-1}$, then $h \in N(T)$ and $hT \cdot (g'T,\phi') = (gT, \phi)$, so the two pairs have the same image in the quotient by the Weyl group and we don't really have an instance of non-injectivity.
  \item If $T \neq h T h^{-1}$, then $T \cap h T h^{-1} = \{\pm I\}$, since geometrically $T$ and $hTh^{-1}$ are two great circles on $SU(2) \cong S^3$. So, in this case, $\phi \colon A \to \{\pm I\}$ which implies in turn that $\phi' = \phi$, but leaves $g$ and $g'$ unconstrained.
  \end{itemize}

  In summary, all the fibres of the right vertical map are singletons except for the fibres above points $\phi \in \Hom(A,\{\pm I\})$, which are of the form $SU(2)/N(T) \times \{\phi\}$. The pushout of the top and left maps in the square precisely collapses each of those fibres down to a single point, so the pushout is the bottom right corner as desired. 
\end{proof}

\begin{remark}
The (homotopy) pushout square for $A=\Z^n$ first appeared in \cite{lowdim} and indeed, the proof above is a straightforward adaptation of our proof there. The homotopy type of $\Hom(\Z^n,SU(2))$ after one suspension has been described independently in \cite{Ad4,Baird} and \cite{Crabb} (in fact, the title of this paper was motivated by the one in \cite{Baird}). To be precise, for $n\geq 2$
\[\Sigma \Hom(\Z^n,SU(2))\simeq\Sigma\left((S^{3})^{\vee n} \vee \bigvee_{k=2}^n 
\left((\RP^2)^{k\lambda_2}/s_k(\RP^2)\right)^{\vee \binom{n}{k}}\right),\] 
where $\lambda_2$ is the canonical line bundle over $\RP^2$, $X^\lambda$ is the Thom space of a vector bundle $\lambda$ over $X$, and $s_k$ is the zero section of $k\lambda_2$. For an arbitrary finitely generated abelian group $A$, the homotopy type of $\Hom(A,SU(2))$ after one suspension is given in more detail in \cite[Example 4.3]{AdemGomez}.
\end{remark}

\begin{exam}\label{hom-Q}
  As an example of a non-finitely generated group, take $A = \Q$ ---with the discrete topology. Since every element of $\Q$ is divisible by $2$, there are no non-trivial homomorphisms $\Q \to \{\pm I\}$, which simplifies the pushout somewhat. The space $\Hom(\Q, T)$ is the Pontryagin dual of $\Q$ as a discrete group, which is well-known to be $\mathbb{A}/\Q$. Here $\mathbb{A}$ is the group of adeles, built out of the $p$-adic numbers $\Q_p$ and $p$-adic integers $\Z_p$, as the subgroup of $\R \times \prod_p \Q_p$ consisting of sequences such that for almost all primes $p$, the $p$-th coordinate lies in $\Z_p$. The homomorphism $\psi_a \colon \Q \to T$ corresponding to $a = (a_{\infty}, (a_p)_p) \in \R \times \prod_p \Q_p$ is $\psi_a(r) = e^{-2\pi i a_{\infty} r} \prod_p e^{- 2\pi i \{a_p r\}_p}$, where $\{x\}_p$ denotes the $p$-adic fractional part of $x \in \Q_p$, i.e., the rational number with $p$-th power denominator that one obtains as the part of the $p$-adic expansion of $x$ coming after the decimal point. From that description of $\psi_a$, one sees that the complex conjugation action on $\Hom(\Q,T)$ corresponds to sending $a$ to $-a$ in $\mathbb{A}/\Q$. Putting this all together, we get that the space $\Hom(\Q, SU(2))$ is homeomorphic to $\left(S^2 \times_{\Z/2} \mathbb{A}/\Q\right)\!/\RP^2$, where the copy of $\RP^2$ inside $S^2 \times_{\Z/2} \mathbb{A}/\Q$ that gets collapsed to a point is $S^2 \times_{\Z/2} \{0\}$.

  By the same token, for $A = \Q^n$ we get $\Hom(\Q^n, SU(2)) \cong \left(S^2 \times_{\Z/2} (\mathbb{A}/\Q)^n\right)\!/\RP^2$, where the copy of $\RP^2$ collapsed to a point is $S^2 \times_{\Z/2} \{(0,\ldots,0)\}$.
\end{exam}

\begin{exam}
  As a more exotic example of Theorem \ref{main}, consider the Heisenberg group over the rationals:
  \[ H_\Q := \left\{
      \begin{pmatrix}
        1 & x & y \\
        0 & 1 & z \\
        0 & 0 & 1 \\
      \end{pmatrix} \middle| x, y, z \in \Q
    \right\}. \]
  There are no epimorphisms $H_\Q \to Q_{2^q}$ for any $q \ge 3$ at all. Indeed, any element in $H_\Q$ has a square root, and the same is not true of some elements in $Q_{2^q}$. Therefore Theorem \ref{main} implies that $\Hom(H_\Q,SU(2)) \cong \Hom(H_\Q^\ab, SU(2))$. We have $H_\Q^\ab \cong \Q \oplus \Q$, so that $\Hom(H_\Q^\ab, SU(2)) \cong \Hom(\Q \times \Q, SU(2))$. From Example \ref{hom-Q}, we get $ \Hom(H_\Q, SU(2)) \cong \left(S^2 \times_{\Z/2} (\mathbb{A}/\Q)^2 \right)\!/\RP^2$.
\end{exam}
\section{Spaces of homomorphisms from $F_n/\Gamma^q_n$ to $SO(3)$ and $U(2)$} \label{sec:so3-u2}

Corollary \ref{Theorem1} can be used to describe the spaces of nilpotent $n$-tuples in $SO(3)$ and $U(2)$. We study first $SO(3)$, which we will identify with $SU(2)/\{\pm I\}=PU(2)$. 

Let $\pi\colon SU(2)\to PU(2)$ denote the quotient homomorphism,  and $\pi_*\colon\Hom(F_n/\Gamma^{q+1}_n,SU(2))\to \Hom(F_n/\Gamma^{{q+1}}_n,PU(2))$ the induced map. We claim that for any $q\geq 2$, the image of $\pi_*$ is precisely $\Hom(F_n/\Gamma^{q}_n,PU(2))$. Clearly any commuting tuple in $SU(2)$ is mapped to a commuting tuple in $PU(2)$. Now, let $H$ be a non-abelian subgroup of $SU(2)$ of nilpotency class $q$. As showed before this implies that $\Gamma^{q}(H)=\{\pm I\}$ and hence $\Gamma^q(\pi(H))=I$. Therefore $\pi(H)$ has nilpotency class $q-1$. It remains to show that any nilpotent subgroup in $PU(2)$ of nilpotency class $(q-1)$ has a lift to a nilpotent subgroup of $SU(2)$ of nilpotency class $q$. This follows from the fact that $\pi$ is an epimorphism and the preimage of any trivial commutator is $\{\pm I\}$, and thus central in $SU(2)$. Therefore, for any $q\geq 2$ 
 \[\pi_*\colon\Hom(F_n/\Gamma^{q+1}_n,SU(2))\to \Hom(F_n/\Gamma^{q}_n,PU(2))\]
 is surjective as claimed.
  
A result of W. M. Goldman \cite[Lemma 2.2]{Goldman} implies that the restriction $\pi_*^{-1}(C)\to C$ to any connected component $C$ of $ \Hom(F_n/\Gamma^{q}_n,PU(2))$, is a $2^n$-fold covering map (the fiber is in bijection with $\Hom(F_n/\Gamma^{q+1}_n,\ker \pi)=\{\pm I\}^n$). From Theorem \ref{Theorem1} we know the connected components of $\Hom(F_n/\Gamma^{q+1}_n,SU(2))$: one is $\Hom(\Z^n,SU(2))$ and the others, which we call $K_i$, are all homeomorphic to $PU(2)$.  We will describe the correspondence between components $K_i$ and the connected components $\pi_*(K_i)$ in $\Hom(F_n/\Gamma^{q}_n,PU(2))$. Each of the components $K_i$ consists of representations (by elements of $PU(2)$) conjugate to a fixed surjective homomorphism $\rho\colon F_n/\Gamma^{q+1}_n\to Q_{2^{r+1}}$ with $2\leq r\leq q$. The image under $\pi_*$ of these components are the conjugated homomorphisms (also by elements of $PU(2)$) of the surjective homomorphism $\pi_*(\rho)\colon F_n/\Gamma^{q}_n\to D_{2^{r}}$.  Since $\pi_*$ is open, $\pi_*|\colon K_i\to \pi_*(K_i)$ is a covering map. Fix a homomorphism $\sigma\colon F_n/\Gamma^{q}_n\to D_{2^{r}}$ in $\pi_*(K_i)$. Then the lifts of $\sigma$ in $K_i$ are among the homomorphisms $g\rho g^{-1}\colon F_n/\Gamma^{q+1}_n\to gQ_{2^{r+1}}g^{-1}$ for $g \in PU(2)$. To have $\pi \circ g\rho g^{-1} = \sigma$, $g$ must be in $Z(D_{2^{r}})$.  We have two different cases. First, if $r\geq3$, then $Z(D_{2^{r}})=\langle \pi(\xi_4)\rangle$ and $\pi_*(K_i)\cong PU(2)/\pi(\mu_4)$, where $\mu_4$ is the cyclic group generated by $\xi_4$ as before. Then the restriction $\pi_*|_{K_i}\colon K_i\to \pi_*(K_i)$ is a $2$-fold covering map. When $r=2$, $Z(D_4)=D_4$, and $\pi|_{K_i}\colon K_i\to \pi_*(K_i)\cong PU(2)/\pi(Q_8)$ is a 4-fold covering map. 

With these covering spaces, we can easily count the connected components of $\Hom(F_n/\Gamma^{q}_n,SO(3))$. Indeed, let $\Hom(\Z^n,SO(3))_{\mathds{1}}$ denote the connected component containing $(I,...,I)$. The connected component corresponding to $\Hom(\Z^n,SU(2))$ is mapped under $\pi_*$ to $\Hom(\Z^n,SO(3))_{\mathds{1}}$. For components of commuting tuples, other than the component $\Hom(\Z^n,SO(3))_{\mathds{1}}$, the correspondence is $\frac{2^{n}}{4}$ to 1, and for components consisting of $n$-tuples that generate a subgroup of nilpotency class at most $r$, with $r\geq2$, is $\frac{2^{n}}{2}$ to 1. Thus, we only need to divide the numbers $C(n,2)$ and $C(n,r)-C(n,2)$ of Corollary \ref{Theorem1} by $2^{n-2}$ and $2^{n-1}$ respectively.

\begin{corl}\label{corl1}
Let $q\geq 2$, $n\geq 2$. Then
\[\Hom(F_n/\Gamma^{q}_n,SO(3))\cong\Hom(\Z^n,SO(3))_{\mathds{1}}\sqcup \bigsqcup_{M(n)} SU(2)/Q_8\sqcup\bigsqcup_{M(n,q)} SU(2)/\mu_{4}, \]
where $M(n)=\frac{(2^n-1)(2^{n-1}-1)}{3}$ and $M(n,q)=(2^n-1)(2^{(q-2)(n-1)}-1)2^{n-2}$.
\end{corl}

\begin{remark}
The case $q=2$ which corresponds to commuting elements was originally computed in \cite{Torres}, and the number $M(n)$ agrees with their calculation.
\end{remark}

Now we discuss the situation for $U(2)$. Any matrix in $X\in U(2)$ can be written as $\sqrt{\det(X)}X^\prime$, where $X^\prime\in SU(2)$. In this decomposition, $[X,Y]=[X^\prime,Y^\prime]$ for any $X,Y$ in $U(2)$. Consider the map
\[(S^1)^n\times \Hom(F_n/\Gamma^{q+1}_n,SU(2))\to\Hom(F_n/\Gamma^{q+1}_n,U(2))\] given by $(\lambda_1,\ldots,\lambda_n,x_1,\ldots,x_n)\mapsto (\lambda_1 x_1,\ldots,\lambda_n x_n)$. By the previous observation this a surjective map. This representation for an $n$-tuple is uniquely determined up to signs, that is, we have a homeomorphism
\[(S^1)^n\times_{(\Z/2)^n} \Hom(F_n/\Gamma^{q+1}_n,SU(2))\cong\Hom(F_n/\Gamma^{q+1}_n,U(2)).\]
Using Theorem \ref{Theorem1} we see that the connected components of this space are homeomorphic to the space $(S^1)^n\times_{(\Z/2)^n}\Hom(\Z^n,SU(2))\cong \Hom(\Z^n,U(2))$ or $(S^1)^n\times_{H}PU(2)$ where $H$ is a subgroup of $(\Z/2)^n$. The first component corresponds to the commuting $n$-tuples in $U(2)$ and the latter to the $n$-tuples that generate a non-abelian subgroup of nilpotency class at most $q$. We can see the $(\Z/2)^n$ action as an action on the indexing set of the connected components, which can be represented by elements of $\Hom(F_n/\Gamma^{q+1}_n,Q_{2^{q+1}})$. To count the number of connected components, let $\vec{\varepsilon}=(\varepsilon_1,\ldots,\varepsilon_n)$ with each $\varepsilon_i = \pm 1$ be an arbitrary element in $(\Z/2)^n$ and $\vec{x}=(x_1,\ldots,x_n)$ a non-commutative $n$-tuple in $\Hom(F_n/\Gamma^{q+1}_n,Q_{2^{q+1}})$. Then the stabilizer of this element consists of 
\[\text{Stab}(\vec{x})=\{\vec{\varepsilon}\mid(\varepsilon_1x_1,\ldots,\varepsilon_nx_n)=(gx_1g^{-1},\ldots,gx_ng^{-1})\text{ for some } g\in SU(2)\}.\] 
That is, such $g$ either commutes or anti-commutes with all $x_i$. We have several cases. Let $\vec{\varepsilon}\in \text{Stab}(\vec{x})$ be a non trivial element. \\

\noindent {\bf Case 1:} Suppose some $x_i$ lies in $T-\{\pm I\}$. 

$\bullet$ If $\varepsilon_i=1$, $x_i=gx_ig^{-1}$ and thus $g$ must also lie in $T$. To generate a non-abelian  nilpotent subgroup of $Q_{2^{q+1}}$ with $x_i$, we need at least one more element of the form $\xi_{2^q}^kw$ for some $k$. This element does not commute with any element of $T-\{\pm I\}$ and only anti-commutes with $\pm \xi_4$. Thus, the only choices for $g$ are $\pm I$ and $\pm\xi_4$.

$\bullet$ If $\varepsilon_i=-1$, $-x_i=gx_ig^{-1}$ and by Lemma \ref{simult-diag}, $x_i=\pm \xi_4$ and $g\in Tw$. Again, in the $n$-tuple there must be an element of the form $x_j=\xi_{2^q}^kw$ for some $k$. The only elements in $Tw$ that commute or anti-commute with $x_j$ are $g=\pm\xi_{2^q}^kw$ or $g=\pm \xi_4\xi_{2^q}^kw$. Note that in this case, the remaining elements in the $n$-tuple can only be of the form $\pm I,\pm\xi_4,\pm\xi_{2^q}^kw$ or $\pm \xi_4\xi_{2^q}^kw$, which generates a copy of $Q_8$. \\

\noindent {\bf Case 2:} Suppose all $x_i$ lie in $Tw \cup \{\pm I\}$.

Not all $x_i$ can be $\pm I$, since $\vec{x}$ is a non-commutative tuple, so some $x_i$ is of the form $x_i=\xi_{2^q}^kw$.

$\bullet$ Suppose $\vec{x}$ generates a nilpotent group of nilpotency class $r\geq 3$. Then there is at least one $x_j=\xi_{2^q}^lw$ that is different from $\pm x_i$ or $\pm\xi_4x_i$. Thus the only choices for $g$ are $\pm I$ and $\pm \xi_4$.

$\bullet$ Suppose $\vec{x}$ generates a nilpotent group of nilpotency class 2. Then the only other choices for the remaining $x_j$'s are $\pm\xi_{2^q}^kw$ or $\pm \xi_4\xi_{2^q}^kw$. Hence $g$ can only be $\pm I,\pm\xi_4,\pm\xi_{2^q}^kw$ or $\pm \xi_4\xi_{2^q}^kw$.\\

We can conclude that if $\vec{x}$ generates a copy of $Q_8$, then $|\text{Stab}(\vec{x})|=4$ and $|\text{Stab}(\vec{x})|= 2$ in any other case.

\begin{corl}\label{corl2}
Let $q\geq 2$ and $n\geq 1$, then
\[\Hom(F_n/\Gamma^{q+1}_n,U(2))\cong \Hom(\Z^n,U(2))\sqcup \bigsqcup_{{M}(n)}(S^1)^n\times_{(\Z/2)^2} PU(2)\sqcup \bigsqcup_{{M}(n,q)}(S^1)^n\times_{(\Z/2)} PU(2),\]
where $M(n)$ and $M(n,q)$ are as in Corollary \ref{corl1}.
\end{corl}

\section{Nil-2 tuples in $U(m)$}\label{sec:um}

What about $m>2$, what are the connected components of $\Hom(F_n/\Gamma^{q+1}_n,U(m))$ then? We can not give an answer to this in its full generality, but we can at least say something about the case $q=2$.

Let ${\bf a}$ be a partition of $\{1,2,\ldots,m\}$ into disjoint non-empty subsets. Define $U({\bf a})$ as the subgroup of $U(m)$ consisting of $m\times m$ ``block diagonal matrices with blocks indexed by ${\bf a}$'', by which we mean matrices $A \in U(m)$ whose $(i,j)$-th entry is $0$ whenever $i$ and $j$ are in different parts of the partition ${\bf a}$. To explain our terminology, notice that when each part of ${\bf a}$ consists of consecutive numbers, say, if the parts are $\{1,\ldots,m_1\}$, $\{m_1+1, \ldots, m_1+m_2\}$, $\{m_1+m_2+1, \ldots, m_1+m_2+m_3\}$, $\ldots$, then $A$ is what is traditionally called a block diagonal matrix:
\[A = \left(\begin{matrix}
      A_1&&0\\
      &\ddots&\\
      0&&A_k
    \end{matrix}\right); \qquad A_i \in U(m_i)\]

The conjugacy class of the subgroup $U({\bf a})$ depends only on the
sizes of the parts of ${\bf a}$. To be specific, if $\pi$ is any
permutation of $\{1, \ldots, m\}$ such that the image of each part of
${\bf a}$ consists of consecutive numbers, then $U({\bf a})$ is
conjugate, via the permutation matrix associated to $\pi$, to the
subgroup of traditional block diagonal matrices as above (where the
$m_i$ are the sizes of the parts of ${\bf a}$). In particular, the
subgroup $U({\bf a})$ is always isomorphic to $\prod_{i=1}^kU(m_i)$
where the $m_i$ are the sizes of the parts of ${\bf a}$ but the isomorphism is far from unique.

Let $Z_{\bf a}$ denote the center of $U({\bf a})$ which consists of ``block scalar matrices'': diagonal matrices $\mathrm{diag}(\lambda_1, \ldots, \lambda_m) \in U(m)$ such that $\lambda_i = \lambda_j$ whenever $i$ and $j$ are in the same part of ${\bf a}$. For example, if the parts consist of consecutive numbers, the elements of $Z_{\bf a}$ are of the form:
\[\left(\begin{matrix}
\lambda_1 I_{m_1}&&0\\
&\ddots&\\
0&& \lambda_k I_{m_k}
\end{matrix}\right).\]

Given any diagonal matrix $D = \mathrm{diag}(\lambda_1, \ldots, \lambda_m) \in U(m)$ there is a coarsest partition ${\bf a}(D)$ such that $D \in Z_{{\bf a}(D)}$, namely, the partition where $i$ and $j$ are in the same part \emph{if and only if} $\lambda_i = \lambda_j$. One can easily check that the centralizer of $D$ is precisely $U({\bf a}(D))$.

Our goal is to interpret 2-nilpotent tuples of $U(m)$ as \emph{almost commuting} elements of the the subgroups $U({\bf a})$. What we mean by ``almost commuting'' is as follows: for any topological group $G$ and any closed subgroup $K\subset G$ contained in the center of $G$, the space of $K$-almost commuting $n$-tuples is defined to be
\[ B_n(G,K) := \{ (x_1,\ldots,x_n) \in G^n \mid [x_i, x_j] \in K \quad (1 \le i, j \le n)\}.\]
Thus, we have an inclusion $B_n(G,Z(G))\subset \Hom(F_n/\Gamma^3_n,G)$.

In particular, for $U(m)$,
\[\bigcup_{{\bf a}\vdash m}B_n(U({\bf a}),Z_{\bf a})\subset \Hom(F_n/\Gamma^3_n,U(m)),\]
where we've borrowed the notation ${\bf a} \vdash m$ typically used for partitions of the \emph{number} $m$ to indicate that the union is over all partitions of the \emph{set} $\{1, \ldots, m\}$ where each parts consists of consecutive numbers and the parts are ordered by size.

Let $\Rep(\Gamma,G)$ denote the orbit space of the action of $G$ on $\Hom(\Gamma,G)$ by conjugation.

\begin{prop}\label{prop3}
  The above inclusion is surjective upon passing to orbits, that is, if \[\pi : \Hom(F_n/\Gamma^3_n,U(m)) \to \Rep(F_n/\Gamma^3_n,U(m))\] denotes the quotient map, then:
\[\pi\left(\bigcup_{{\bf a}\vdash m}B_n(U({\bf a}),Z_{\bf a})\right) = \Rep(F_n/\Gamma^3_n,U(m)).\]
\end{prop}

\begin{proof}
  Let $(x_1,\ldots,x_n)$ be an element of $\Hom(F_n/\Gamma^3_n,U(m))$. Then every commutator $[x_i,x_j]$ is central in the group generated by $\{x_1,\ldots,x_n\}$. In particular, all commutators commute with each other. Since each $x_i$ is in $U(m)$ and hence diagonalizable, we can simultaneously diagonalize all commutators by an element $g\in U(m)$. Let $y_i:=gx_ig^{-1}$ and $y_{ij}:=g[x_i,x_j]g^{-1}$. Now, each $y_{ij}$ lies in the center $Z_{{\bf a}_{ij}}$ for some coarsest partition ${\bf a}_{ij}$. Choose ${\bf a}$ as the infimum of all the ${\bf a}_{ij}$, that is, as the coarsest partition refining all ${\bf a}_{ij}$. We have by construction $y_{ij} \in Z_{{\bf a}_{ij}} \subseteq Z_{{\bf a}}$; and for each $k$, we have that $y_k$ is in the centralizer of each $y_{ij}$, so $y_k \in \bigcap_{ij} U({\bf a}_{ij}) = U({\bf a})$.

  The last remaining detail is that this partition ${\bf a}$ may not have parts that consist of consecutive numbers, or those parts may not be ordered by size, but, as explained above, a further conjugation fixes that.
\end{proof}

\begin{remark}
  The same argument works for $SU(m)$ and its subgroups $SU({\bf a})$. There are only two minor differences: the first is that $SU({\bf a})$ is identified with the subgroup of matrices in $\prod_{i=1}^k U(m_i)$ of determinant 1; the second is that for the last bit of the proof, the ``consecutivization'', one needs to observe that for every partition one can always find an \emph{even} permutation such that the image of each part consists of consecutive numbers and permutation matrices for even permutations lie in $SU(m)$.
\end{remark}

\begin{exam} As an application of Proposition \ref{prop3}, we calculate the number of connected components of $\Rep(F_n/\Gamma^3_n,U(3))$. The union is indexed by the partitions $3$, $1+2$ and $1+1+1$, where we are using partitions of the number $3$ as shorthand for partitions of the set $\{1,2,3\}$ whose parts are of the given sizes and consist of consecutive numbers. So $\Rep(F_n/\Gamma^3_n, U(3))$ is the union of the images under $\pi$ of $B_n(U(3),S^1)$, $B_n(U(1) \times U(2),\{1\} \times S^1) \cong U(1)^n \times B_n(U(2), S^1)$, and $(U(1)^3)^n$. The intersection of those first two $B_n$'s is clearly the last one, but not only that: we claim the same is true for  their images under $\pi$. This is clear once we characterize the $n$-tuples $x = (x_1, \ldots, x_n) \in \Hom(F_n/\Gamma^3_n,U(3))$ such that $\pi(x) \in \pi(B_n(U(\mathbf{a}), Z_{\mathbf{a}}))$ by means of a conjugation-invariant property shared by the commutators $[x_i, x_j]$ ($ 1\le i,j \le n$), in each of those three cases:

  \begin{itemize}
  \item For $\mathbf{a} = 3$, all commutators are scalar matrices.
  \item For $\mathbf{a} = 1+2$, all commutators share an eigenvector with eigenvalue 1.
  \item For $\mathbf{a} = 1+1+1$, all commutators are the identity.
  \end{itemize}
  
  We conclude that:
  \[\Rep(F_n/\Gamma^3_n,U(3))=B_n(U(3),S^1)/U(3)\cup_{(U(1)^3)^n/U(3)}(U(1)^n\times B_n(U(2),S^1))/U(3).\]
  With this one can count the number of connected components. First, recall that the quotient map $U(3)\to U(3)/S^1=PU(3)$ induces a principal bundle  $(S^1)^n\to B_n(U(3),S^1)\to \Hom(\Z^n,PU(3))$; so $B_n(U(3),S^1)$ has the same number of connected components as $ \Hom(\Z^n,PU(3))$, which was calculated in \cite[Theorem 1]{AdCoGo}. Since $B_n(U(2),S^1)=\Hom(F_n/\Gamma^3_n,U(2))$, and passing to orbits induced by conjugation does not change the number of components (because $U(3)$ is path connected), then the total number of connected components is
\[\frac{(3^n-1)(3^{n-1}-1)}{8}+\frac{(2^{n}-1)(2^{n-1}-1)}{3}+1.\]
\end{exam}

\begin{remark}
The spaces $\Hom(\Gamma,U(m))$ where $\Gamma$ is a central extension of the form $1\to \Z^r\to \Gamma\to \Z^n\to 1$ were studied in \cite{AdemCheng}.  The case where $r=1$, has a very detailed description, and they give a formula for the number of connected components for the example case $\Hom(H_\Z,U(m))$, where $H_\Z$ is the Heisenberg group (over the integers). Taking $n=2$ in our previous example, says that the space of homomorphisms from the Heisenberg group $H_\Z \cong F_2/\Gamma^3_2$ to $U(3)$ has 4 connected components, which agrees with their formula when $m=3$.
\end{remark}

\section{Cohomology of $B(r,Q_{2^n})$ for $r=2$ and $3$}

Now we turn our attention to the classifying spaces of principal $G$-bundles of transitional nilpotency class less than $q$, for a topological group $G$. As described in \cite{Ad5}, the spaces $\Hom(F_n/\Gamma^q_n,G)\subset G^n$ give rise to a simplicial subspace of the nerve of $G$. The geometric realizations $B(q,G):=|\Hom(F_*/\Gamma^q_*,G)|$ fit into a natural filtration of $BG$
\[\Bcom G:=B(2,G)\subset B(3,G)\subset\cdots\subset B(q,G)\subset\cdots\subset BG.\] So far we have completely described the subgroups of $Q_{2^{q+1}}$. This will allow us to compute the homotopy type of $B(r,Q_{2^{q+1}})$
as follows. Let $G$ be a finite group. Consider the poset $\mathcal{N}_r(G)$ of all subgroups of $G$ of nilpotency class less than $r$, ordered by inclusion, and the subposet $\mathcal{P}_r(G) = \{M_\alpha\} \cup \{M_{\alpha}\cap M_\beta\}$ where the $M_\alpha$ are the maximal subgroups in $\mathcal{N}_r(G)$. It was proved in \cite[Theorem 4.6]{Ad5} that when $\mathcal{P}_r(G)$ is a tree, there is a homotopy equivalence
\[B(r,G)\simeq B\left(\colim_{A\in \mathcal{N}_r(G)}A\right).\]

Let $q\geq2$. By Lemma \ref{Lemma7} we conclude that for $2\leq r\leq q$ the maximal subgroups of $Q_{2^{q+1}}$ of nilpotency class less than $r$ are $\mu_{2^{q}}$ and the $2^{q+1-r}$ subgroups isomorphic to $Q_{2^{r}}$; any two of them have the same intersection, namely, $\mu_{2^{r-1}}$. Therefore we have:
\[B(r,Q_{2^{q+1}})\simeq B\left(\bigast_{\mu_{2^{r-1}}}^{2^{q+1-r}}Q_{2^{r}} *_{\mu_{2^{r-1}}} \mu_{2^{q}}\right)\]
where $*$ denotes the amalgamated product of groups.


\paragraph{Cohomology of $\Bcom Q_{2^n}$.} Taking $r=2$ we see that $\Bcom Q_{2^n}\simeq B(\Z/2^{n-1}*_{\Z/2}(\bigast_{\Z/2}^{2^{n-2}}\Z/4))$ for any $n\geq 3$. Applying the associated Mayer-Vietoris sequence inductively we obtain 
\[H^i(\Bcom Q_{2^n};\Z)\cong\left\{\begin{matrix}
\Z &i=0\\
\Z/2^{n-1}\oplus(\Z/2)^{2^{n-2}}&i \text{  even}\\
0&\text{otherwise}.
\end{matrix}\right.\]

\begin{remark}
  The integral homology groups $H_i(\Bcom Q_8;\Z)$ were originally computed in \cite[Example~8.12]{Ad5}, where the authors were studying the homology of the spaces $\Bcom A$, where $A$ is a \emph{transitively commutative group} (that is, a group where all centralizers of non-identity elements are abelian; also know as a CA-group).
\end{remark}

\begin{prop}\label{cr2gq}\leavevmode
\begin{enumerate}
\item There is an isomorphism of graded rings $H^*(\Bcom Q_8;\F_2)\cong \F_2[y_1,y_2,y_3,z]/(y_iy_j,y_1^2+y_2^2+y_3^2,i\ne j)$ where $y_i$ has degree 1 and $z$ degree 2. 

\item Let $n\geq 4$. Then the $\F_2$-cohomology ring of $\Bcom Q_{2^{n}}$ has a presentation given by
\[H^*(\Bcom Q_{2^{n}};\F_2)\cong \F_2[x_1,x_2,y_1,\ldots,y_{2^{n-2}},z]/(x_1^2,x_ky_i,y_iy_j,i\ne j ,x_2+\sum_{i=1}^{2^{n-2}}y_i^2),\]
where $x_1,y_j$ have degree 1 and $x_2,z$ have degree 2.
\end{enumerate}
\end{prop}

\begin{proof}
We work out the case $n\geq 4$. Let $\Gamma=\Z/2^{n-1}*_{\Z/2}(\bigast_{\Z/2}^{2^{n-2}}\Z/4)$. We use the central extension 
\[\Z/2\vartriangleleft \Gamma\to \Z/2^{n-2}*\bigast^{2^{n-2}}\Z/2.\] 
Recall that for $j>1$, $H^*(B\Z/2^j;\F_2)\cong \F_2[x_1,x_2]/x_1^2$ where $\deg(x_1)=1$ and $\deg(x_2)=2$. Thus, $H^*(B(\Z/2^{n-2}*\bigast^{2^{n-2}}\Z/2);\F_2)\cong \F_2[x_1,x_2,y_1,\ldots,y_{2^{n-2}}]/(x_1^2,x_ky_i,y_iy_j,i\ne j )$. The $k$-invariant of the associated Serre spectral sequence of this extension is $x_2+\sum_{i=1}^{2^{n-2}}y_i^2$. Therefore the $E^{*,*}_3$ page is 
\[\F_2[z]\otimes\F_2[x_1,x_2,y_1,\ldots,y_{2^{n-2}}]/(x_1^2,x_ky_i,y_iy_j,i\ne j ,x_2+\sum_{i=1}^{2^{n-2}}y_i^2).\]
The Steenrod square $\Sq^1(x_2+\sum_{i=1}^{2^{n-2}}y_i^2)=\Sq^1(x_2)$ is $0$ since it can be expressed only in terms of $x_2$. That is, $d_3=0$ and thus the spectral sequence abuts to $E^{*,*}_3$.

So we have found the $E_\infty$-page along with its ring structure. Since all the relations involve only generators from the base of the fibration, these relations hold in the cohomology ring as well.
\end{proof}

\begin{remark}
  Recall that $H^*(BQ_8;\F_2)=\F_2[x,y,t]/(x^2+xy+y^2,x^2y+xy^2)$ where $x,y$ have degree 1 and $t$ has degree 4. Consider the inclusion $\iota\colon \Bcom Q_8\to BQ_8$ and $\iota^*\colon H^*(BQ_8;\F_2)\to H^*(\Bcom Q_8;\F_2)$. By Proposition \ref{cr2gq}, $\iota^*(x^2y)=0$ since $\iota^*(x)$ and $\iota^*(y)$ are linear combinations of the $y_j$ and all monomials of degree 3 in the $y_j$ are $0$. Thus $\iota^*$ is not injective (this was also pointed out in \cite[Proposition 6.8]{Okay}). It is unclear for which groups $G$ and with which coefficients $\iota^* : H^*(\Bcom G) \to H^*(BG)$ is injective. In \cite{lowdim} it is shown that $\iota^*$ is injective both with $\F_2$ and with integral coefficients for the groups $SU(2), U(2), O(2), SO(3)$.
\end{remark}

\paragraph{Cohomology of $B(3,Q_{16})$.} By the above discussion, $B(3, Q_{16})$ is the classifying space of the group $\Gamma := Q_8 \ast_{\mu_4} Q_8 \ast_{\mu_4} \mu_8$. To compute it's $\F_2$ cohomology, we'll use the Lyndon-Hochschild-Serre spectral sequence for the central extension $1 \to \mu_2 \to \Gamma \to \Gamma/\mu_2 \to 1$, so we first need the cohomology of the group $\Gamma/\mu_2 = (Q_8/\mu_2) \ast_{\mu_4/\mu_2} (Q_8/\mu_2) \ast_{\mu_4/\mu_2} (\mu_8/\mu_2) \cong (\Z/2 \times \Z/2) \ast_{\Z/2} (\Z/2 \times \Z/2) \ast_{\Z/2} \Z/4$. We are lucky that the first amalgamated product simplifies: $(\Z/2 \times \Z/2) \ast_{\Z/2} (\Z/2 \times \Z/2) \cong (\Z/2 \ast \Z/2) \times \Z/2$ ---this is easiest to see if we use the shear automorphism of the $\Z/2 \times \Z/2$'s to replace the diagonal embeddings of $\Z/2$ with the inclusions into, say, the first factor. Under this isomorphism, the inclusion $\Z/2 \to (\Z/2 \times \Z/2) \ast_{\Z/2} (\Z/2 \times \Z/2)$ corresponds to the inclusion $\Z/2 \to (\Z/2 \ast \Z/2) \times \Z/2$ into the second factor.

To compute the cohomology of $\Gamma/\mu_2$, we can now use the Mayer-Vietoris sequence for the homotopy pushout square
\[\xymatrix{B\Z/2 \ar[r] \ar[d] & B\left((\Z/2 \ast \Z/2) \times \Z/2\right) \ar[d] \\
    B\Z/4 \ar[r] & B(\Gamma/\mu_2).}\]
The top map is surjective on cohomology which implies all the connecting homomorphisms are zero; this in turn means that the map $H^*(B(\Gamma/\mu_2)) \to H^*(B\left((\Z/2 \ast \Z/2) \times \Z/2\right) \vee B\Z/4)$ is an injective ring homomorphism. The groups in the pushout (and thus their wedge) have the following cohomology rings, where $t_2$ has degree $2$ and all other variables have degree $1$:
\begin{align*}
  H^*(B\left((\Z/2 \ast \Z/2) \times \Z/2\right)) & = \F_2[y_1, y_2, t_1]/(y_1y_2) \\
  H^*(B\Z/4) & = \F_2[y_3, t_2]/(y_3^2) \\
  H^*(B\left((\Z/2 \ast \Z/2) \times \Z/2\right) \vee B\Z/4) & = \F_2[y_1, y_2, t_1,y_3,t_2]/(y_1y_2,y_3^2,uv),
\end{align*}
where $u$ ranges over $\{y_1,y_2,t_1\}$, and $v$ over $\{y_3,t_2\}$.

If $t$ denotes the cohomology generator of the $B\Z/2$ in the top left corner, then the top horizontal map on cohomology is given by $y_1, y_2 \mapsto 0$, $t_1 \mapsto t$, and the left vertical map by $y_3 \mapsto 0$, $t_2 \to t^2$. Therefore, in the Mayer-Vietoris sequence, the kernel of the homomorphism \[H^*(B\left((\Z/2 \ast \Z/2) \times \Z/2\right)) \oplus H^*(B\Z/4) \to H^*(B\Z/2)\] is generated by the linear subspace spanned by $\{(t_1^{2k}, t_2^k) : k \ge 0\}$ together with $K_1 \times K_2$, where $K_1 = (y_1,y_2)$ is the kernel of the ring homomorphism $H^*(B\left((\Z/2 \ast \Z/2) \times \Z/2\right)) \to H^*(B\Z/2)$ and $K_2 = (y_3)$ is the kernel of the ring homomorphism $H^*(B\Z/4) \to H^*(B\Z/2)$. From this one can verify that $H^*(B\Gamma/\mu_2)$ is the subring of $\F_2[y_1, y_2, t_1, y_3, t_2]/(y_1y_2, y_3^2, uv)$ generated by $y_1, y_2, y_3, \beta_1 := y_1 t_1, \beta_2 := y_2 t_1, z:=t_1^2 + t_2$. This subring has the following presentation:
\[H^*(B(\Gamma/\mu_2)) = \F_2[y_1,y_2,y_3,\beta_1,\beta_2,z]/(y_i y_j,y_3^2, y_i \beta_j, \beta_1 \beta_2, \beta_j^2 - y_j^2 z),\]
where $i$ and $j$ range over all indices such that  $i \in \{1,2,3\}$, $j \in \{1,2\}$ and $i \neq j$. Here $y_i$ has degree $1$, and $z$ and the $\beta_j$ have degree $2$.

\smallskip

Now we can go back to the group $\Gamma$. The $E_2$-page of the Lyndon-Hochschild-Serre spectral sequence for the central extension $1 \to \mu_2 \to \Gamma \to \Gamma/\mu_2 \to 1$ is $\F_2[u] \otimes_{\F_2} H^*(B\Gamma/\mu_2)$ and $d_2(u)$ is given by the $k$-invariant of the extension, which we must figure out now. This $k$-invariant must restrict to the correct $k$-invariant for the $Q_8$ summands and for the $\mu_8$ summand. The $k$-invariant for the extension $1 \to \mu_2 \to Q_8 \to \Z/2 \times \Z/2 \to 1$ is given by $x^2 + xy + y^2$ where $x$ and $y$ are the generators of the cohomology of $\Z/2 \times \Z/2$. And the $k$-invariant for the extension $1 \to \mu_8 \to \Z/4$ is $t_2$, again using the presentation $H^*(B\Z/4) = \F_2[y_3, t_2]/(y_3^2)$. To have those restrictions to the extension for the 3 summands, we must have $d_2(u) = z + \beta_1 + y_1^2 + \beta_2 + y_2^2$.

We are in the situation where a page of the spectral sequence is polynomials in one variable, namely $u$, over the base ring $A$ (here, $A:=H^*(B\Gamma/\mu_2)$), in this situation the entire differential is determined by $d_2(u)$ and the next page is easily written in terms of the ideal generated by $d_2(u)$ and its annihilator. Indeed, given an element $a u^q \in E_2^{p,q}$, we have $d_2(au^q) = q a d_2(u) u^{q-1}$, from which we see that (1) the element is in the kernel of $d_2$ if either $q$ is even or $q$ is odd and $a d_2(u) = 0$, and (2) the image of $d_2$ is spanned by monomials $b u^j$ where $b$ is a multiple of $d_2(u)$ and $j$ is even. Therefore, $E_3^{*,*} = A/(d_2(u))[u^2] \oplus u \ann_A(d_2(u))[u^2]$, where the multiplication is given by the canonical $A/(d_2(u))$-module structure on $\ann_A(d_2(u))$.

Now we claim that $d_2(u)$ is not a zero divisor in the ring $A$, so that $\ann_A(d_2(u)) = 0$. This should be plausible because of the ``leading'' term $z$, and we verified this claim using \textsc{Singular} in the appendix. That tells us that the $E_3$-page is also polynomials in $u^2$ over the new base, $A/(d_2(u))$ and to continue we need to compute $d_3(u^2)$. By the Kudo transgression theorem, $d_3(u^2)$ is the image of $\Sq^1(d_2(u))$ from the $E_2$ page. We can compute this Steenrod square by restricting to the $Q_8$ and $\mu_8$ summands. For the $Q_8$'s, using the notation above, $\Sq^1(x^2 + xy +y^2) = x^2y + xy^2$ which is the image of $y_j z + y_j \beta_j$. For the $\mu_8$ we have $\Sq^1(t_2) = 0$. Together these facts tells us $\Sq^1(d_2(u))$ must be $y_1z + y_1\beta_1 + y_2z + y_2\beta_2$. This \emph{is} a zero divisor, even on the $E_2$-page, since it is annihilated by $y_3$. In the appendix we use \textsc{Singular} to verify that $y_3$ generates the entire annihilator in the $E_3$-page (this is even true in the $E_2$-page). Thus, we get that the $E_4$-page is:
\[ E_4^{*,*} = A/(d_2(u), \Sq^1(d_2(u)))[u^4] \oplus u^2 (y_3)[u^4],\]
where the ideal $(y_3)$ is taken in the ring $A/(d_2(u))$.

The $E_5$-page is the same as the $E_4$-page because $d_4$ vanishes for degree reasons on ring generators ($u^4$, $u^2 y_3$) of the above ring. Again using the Kudo transgression theorem, the next possible differential is $d_5(u^4) = \Sq^2 \Sq^1 d_2(u)$. By a similar computation as for $\Sq^1$, we get $d_5(u^4) = y_1 z^2 + y_1^3 \beta_1 + y_2 z^2 + y_2^3 \beta_2$. This, fortunately, is $0$ on the $E_3$-page as we verify in the appendix. Now the spectral sequence collapses: for degree reasons all remaining differentials are zero on the ring generators for the $E_4$-page and thus on their entire pages of definition. As a result, the $E_4^{*,*}$ obtained above gives $H^*(B(3,Q_{16}); \F_2)$ as an $\F_2$-vector space. We are not claiming that the ring structure we found for the $E_4$-page is the correct cohomology ring for $B(3,Q_{16})$ as we were unable to decide the multiplicative extension problems.

\section{Homotopy type of $B(q,SU(2))$}

We finish with a homotopy pushout description of the spaces $B(q,SU(2))$. For $q\geq 2$ consider the map $PU(2)\times (Q_{2^{q+1}})^n\to \Hom(F_n/\Gamma^{q+1}_n,SU(2))$ where $(g,x_1,\ldots,x_n)\mapsto(gx_1g^{-1},\ldots,gx_ng^{-1})$. It is well defined in the quotient \[PU(2)\times_{N_{q+1}}(Q_{2^{q+1}})^n\to \Hom(F_n/\Gamma^{q+1}_n,SU(2))\]
where the normalizer $N_{q+1}:=N_{PU(2)}(D_{2^{q}})$ acts by translation on $PU(2)$ and by conjugation on $(Q_{2^{q+1}})^n$. Consider the subsets $\Epi(F_n/\Gamma^{q}_n,Q_{2^{q+1}}),\Hom(F_n/\Gamma^{q}_n,Q_{2^{q+1}})\subset (Q_{2^{q+1}})^n$ and the restrictions of the above map to the respective subspaces
\[\xymatrix{
\emptyset\ar[r]\ar[d]&PU(2)\times_{N_{q+1}}\Hom(F_n/\Gamma^{q}_n,Q_{2^{q+1}})\ar@{^(->}[d]\ar[r]&\Hom(F_n/\Gamma^{q}_n,SU(2))\ar@{^(->}[d]\\
PU(2)\times_{N_{q+1}}\Epi(F_n/\Gamma^{q}_n,Q_{2^{q+1}})\ar@{^(->}[r]&PU(2)\times_{N_{q+1}}(Q_{2^{q+1}})^n\ar[r]&\Hom(F_n/\Gamma^{q+1}_n,SU(2))
.}\]
We claim that all the above squares are pushouts. From the proof of Theorem \ref{main} one can deduce the homeomorphism
\[PU(2)\times\Epi(F_n/\Gamma^{q}_n,Q_{2^{q+1}})/N_{q+1}\cong \Hom(F_n/\Gamma^{q+1}_n,SU(2))-\Hom(F_n/\Gamma^{q}_n,SU(2)).\]
Since $\Epi(F_n/\Gamma^{q}_n,Q_{2^{q+1}})$ is discrete; $N_{q+1}$ is finite and acts freely, we have that
\[PU(2)\times\Epi(F_n/\Gamma^{q}_n,Q_{2^{q+1}})/N_{q+1}\cong PU(2)\times_{N_{q+1}}\Epi(F_n/\Gamma^{q}_n,Q_{2^{q+1}})\]
which proves the outside square to be a pushout in sets. But, $\Hom(F_n/\Gamma^{q}_n,SU(2))$ is closed and $PU(2)\times_{N_{q+1}}\Epi(F_n/\Gamma^{q}_n,Q_{2^{q+1}})$ is the image of a compact space, so that $\Hom(F_n/\Gamma^{q+1}_n,SU(2))$ has the disjoint union topology. This proves our claim for the outside square, and a similar argument can be used for the square on the left side. Now we look at the right square. The outside and left squares being pushouts imply the right square is also a pushout. The middle arrow is a closed cofibration, and thus the right square is also a homotopy pushout. Moreover, since the maps are either inclusions or given by conjugation, all arrows are simplicial maps. Geometric realization commutes with colimits and homotopy colimits, and hence
\[\xymatrix{
PU(2)\times_{N_{q+1}} B({q},Q_{2^{q+1}})\ar[d]\ar[r]&B({q},SU(2))\ar[d]\\
PU(2)\times_{N_{q+1}}BQ_{2^{q+1}}\ar[r]&B(q+1,SU(2))
}\]
is a pushout of topological spaces and a homotopy pushout.

\begin{theorem}\label{main2}
The inclusions in the filtration
\[\Bcom SU(2)\subset B(3,SU(2))\subset\cdots\subset B(q,SU(2))\subset\cdots\]
are homology isomorphisms with coefficients over a ring $R$ where $2$ is invertible. Moreover, the cohomology ring of each term of the filtration is given by $R[c_2,y_1]/(y_1^2)$, where $\deg(c_2)=\deg(y_1)=4$, and $c_2$ is the pullback of the second Chern class $c_2\in H^4(BSU(2);R)$ under the inclusion $B(q,SU(2))\to BSU(2)$.
\end{theorem}

\begin{proof}
Since we have shown that $B(q, Q_{2^{q+1}})$ is the classifying space of an amalgamated product of finite $2$-groups, its homology with coefficients in any ring in which $2$ is invertible, vanishes, as does that of $BQ_{2^{q+1}}$. This implies that over such a ring the left vertical arrow is a homology isomorphism and thus so is the right vertical arrow. The second part of the statement follows from the integral cohomology ring computations in \cite[Theorem 4.2]{lowdim}
\end{proof}

\begin{remark}
  The filtration $B(2,SU(2)) \subseteq B(3,SU(2)) \subseteq \cdots \subseteq BSU(2)$ is not exhaustive. Indeed, by the previous proposition, the union $\bigcup_{q\ge 2} B(q,SU(2))$ has the rational homology of any of the terms in the union, which differs from the rational homology of $BSU(2)$. This is to be expected, since $SU(2)$ has finitely generated subgroups (even finite subgroups) which are not nilpotent.
\end{remark}

Finally, here are some relations with the classifying spaces $B(q,SO(3))$ and $B(q,U(2))$. Let $q\geq 2$ and consider the map $\pi_n:=\pi_*\colon\Hom(F_n/\Gamma^{q+1}_n,SU(2))\to \Hom(F_n/\Gamma^{q}_n,SO(3))$. In section 3 we said that $\pi_n$ is a $2^n$-fold covering map for every $n\geq0$. Moreover, we have the pullback diagrams
\[\xymatrix{\Hom(F_n/\Gamma^{q+1}_n,SU(2))\ar@{^(->}[r]\ar[d]^{\pi_n}&SU(2)^n\ar[d]^{(\pi)^n}\\
\Hom(F_n/\Gamma^{q}_n,SO(3))\ar@{^(->}[r]&SO(3)^n}\quad\text{and}\quad
\xymatrix{\Hom(\Z^n,SU(2))\ar@{^(->}[r]\ar[d]^{\pi_n}&SU(2)^n\ar[d]^{(\pi)^n}\\
\Hom(\Z^n,SO(3))_{\mathds{1}}\ar@{^(->}[r]&SO(3)^n.}\]

Again, all maps in the squares are simplicial. Since geometric realization of simplicial spaces commutes with taking pullbacks, we obtain pullback squares of topological spaces:
\[\xymatrix{B(q+1,SU(2))\ar[r]\ar[d]&BSU(2)\ar[d]\\
B({q},SO(3))\ar[r]&BSO(3)}\quad\text{and}\quad
\xymatrix{\Bcom SU(2)\ar[r]\ar[d]&BSU(2)\ar[d]\\
  \Bcom SO(3)_{\mathds{1}}\ar[r]&BSO(3).}\]
The right hand side arrow is a $B\Z/2$-bundle, hence the two arrows on the left are also principal $B\Z/2$-bundles. 

\begin{corl}\label{corl3} Let $q\geq 2$. Then the maps induced by the double cover $SU(2)\to SO(3)$ 
\begin{enumerate}
\item  $B(q,SU(2))\to B(q,SO(3))$, and 
\item  $\Bcom SU(2)\to \Bcom SO(3)_{\mathds{1}}$ 
\end{enumerate}
are homology isomorphisms with coefficients over a ring $R$ in which $2$ is invertible. 
\end{corl}
\begin{proof}
Consider the composition $B(q,SU(2))\to B(q+1,SU(2))\to B(q,SO(3))$. By Theorem \ref{main2}, the first arrow is an $H_*(-;R)$ isomorphism, and the above argument shows that  $B(q+1,SU(2))\to B(q,SO(3))$ is an $H_*(-;R)$ isomorphism as well, since $H_*(B\Z/2;R)=H_*(\text{pt};R)$.
\end{proof}

\begin{remark}
The case $q=2$ of Corollary \ref{corl3} was recently proved in \cite[Theorem 4.7]{OkayW}
\end{remark}

Now, let $q\geq 2$ and consider the group homomorphism $\psi\colon S^1\times SU(2)\to U(2)$ given by $(\lambda,x)\mapsto \lambda x$, where $S^1$ represents the scalar matrices in $U(2)$. The kernel of $\psi$ is the group of order $2$ generated by $(-1,-I)$. For each $n$ we have the pullback square 
\[\xymatrix{(S^1)^n\times\Hom(F_n/\Gamma^q_n,SU(2))\ar@{^(->}[r]\ar[d]^{\pi_n}&(S^1\times SU(2))^n\ar[d]^{(\psi)^n}\\
\Hom(F_n/\Gamma^{q}_n,U(2))\ar@{^(->}[r]&U(2)^n}\]
which for the same reasons as before, after taking geometric realizations gives a $B\Z/2$-bundle
\[B\Z/2\to BS^1\times B(q,SU(2))\to B(q,U(2)).\]
Under the same hypotheses as in Corollary \ref{corl3}, we have that $BS^1\times B(q,SU(2))\to B(q,U(2))$ is an $H_*(-;R)$ isomorphism.

\appendix
\section{Computing annihilators in \textsc{Singular}}

We made two claims about annihilators in the calculation of $H^*(B(3,Q_{16}))$: that $\ann_A(k) = 0$ and  $\ann_{A/(k)}(y_1z + y_1\beta_1 + y_2z + y_2\beta_2) = (y_3)$ where $A = H^*(B\Gamma/\mu_2)$ and $k = z + \beta_1 + y_1^2 + \beta_2 +y_1^2$ is the $k$-invariant of the central extension $1 \to \mu_2 \to \Gamma \to \Gamma/\mu_2 \to 1$. We verified these claims with help from the \textsc{Singular} computer algebra system \cite{Singular}. We include the code used for the benefit of potential readers facing similar computations but unfamiliar with what computer algebra systems can offer and how easy to use they can be.
\begin{verbatim}
ring P = 2, (y1,y2,y3,b1,b2,z), dp;
ideal I = y1*y2, y1*y3, y2*y3, y3^2, y2*b1, y3*b1, y1*b2, y3*b2,
          b1*b2, b1^2-y1^2*z, b2^2-y2^2*z;
qring A = std(I);
poly k = z + b1 + y1^2 + b2 + y2^2;
quotient(0, k);
\end{verbatim}

This defines $P$ as the polynomial ring $\F_2[y_1,y_2,y_3,\beta_1, \beta_2,z]$; the $2$ specifies the coefficients are from $\F_2$. Then it defines $A$ as the quotient $P/I$ (when $A$ is defined, $P$ is the ``current ring'' and thus the quotient is implicitly a quotient of $A$), and asks for generators of the ideal $(0:(k))$. \textsc{Singular} responds \verb|_[1]=0|, indicating that $k$ is not a zero divisor. To compute the annihilator of $y_1 z + y_1 \beta_1 + y_2 z + y_2 \beta_2$ in $A/(k)$, we can do:
\begin{verbatim}
qring A_mod_k = std(k);
quotient(0, y1*z + y1*b1 + y2*z + y2*b2);
\end{verbatim}
\textsc{Singular} responds \verb|_[1]=y_3| indicating the annihilator is generated by $y_3$.

We also needed to show that $y_1z^2 + y_1^3b_1 + y_2z^2 + y_2^3b_2$ is zero in $A/(k)$. We can check this by asking for \verb|reduce(y1*z^2 + y1^3*b1 + y2*z^2 + y2^3*b2, std(z^3));| when \verb|A_mod_k| is the current ring.

{\footnotesize {\sc \noindent Omar Antol\'{\i}n Camarena\\
   Instituto de Matem\'aticas, UNAM, Mexico City.}\\
  \emph{E-mail address}:
  \href{mailto:omar@matem.unam.mx}{\texttt{omar@matem.unam.mx}}}
\medskip

{\footnotesize {\sc \noindent Bernardo Villarreal \\
   Instituto de Matem\'aticas, UNAM, Mexico City.}\\
  \emph{E-mail address}:
  \href{mailto:villarrealr@matem.unam.mx}{\texttt{villarreal@matem.unam.mx}}}

\begin{thebibliography}{20}

\small 

\bibitem{AdemCheng} A. Adem and M. Cheng. Representation spaces for central extensions and almost commuting unitary matrices. \emph{J. London Math. Soc.} (2016) 94 (2), 503-524 .

\bibitem{Ad5} A. Adem, F. Cohen and E. Torres-Giese. Commuting elements, simplicial spaces and filtrations of classifying spaces. \emph{Math. Proc. Cambridge Philos. Soc.} \textbf{152} (2012), 91-114. 

\bibitem{AdCoGo} A. Adem, F. Cohen and J. M. G\'{o}mez. Commuting elements in central products of special unitary groups. \emph{Proc. Edinburgh Math. Soc. (2)}, \textbf{56} (2013), no. 1, 1-12. 

\bibitem{Ad4} A. Adem, F. Cohen and J. M. G\'{o}mez. Stable splittings, spaces of representations and almost commuting elements in Lie groups. \emph{Math. Proc. Cambridge Philos. Soc.} \textbf{149} (2010), 455-490. 


\bibitem{AdemGomez} A. Adem and J. M. G\'{o}mez.  On the Structure of Spaces of Commuting Elements in Compact Lie Groups.  \emph{Configuration Spaces: Geometry, Topology and Combinatorics, Publ. Scuola Normale Superiore} {\bf{14}} (CRM Series), 2012 (Birkhauser), 1-26.

\bibitem{Ad3} A. Adem, J. M. G\'{o}mez, J. Lind and U. Tillman. Infinite loop spaces and nilpotent K-theory. \emph{Algebr. Geom. Topol.} \textbf{17} (2017) 869-893.

\bibitem{lowdim} O. Antol\'{\i}n-Camarena, S. Gritschacher, B. Villarreal. Classifying spaces for commutativity of low-dimensional Lie groups. Preprint (2018). \href{https://arxiv.org/abs/1802.03632}{\texttt{{https://arxiv.org/abs/1802.03632}}}.
  
\bibitem{Baird} T. Baird, L. Jeffrey and P. Selick. The space of commuting $n$-tuples in $SU(2)$. \emph{Illinois J. Math.}, \textbf{55}, (2011), no.3, 805-813. 

\bibitem{Bergeron} M. Bergeron. The Topology of Nilpotent Representations in Reductive Groups and their Maximal Compact Subgroups. \emph{Geom. Topol.} \textbf{19-3} (2015), 1383--1407.

\bibitem{Crabb} M. C. Crabb. Spaces of commuting elements in $SU(2)$. \emph{Proc. Edinb. Math. Soc. (2)} \textbf{54} (2011), no. 1, 67-75. 

\bibitem{Singular} W. Decker, G.-M. Greuel, G. Pfister and H. Sch{\"o}nemann. \newblock {\sc Singular} {4-1-0} --- {A} computer algebra system for polynomial computations. \newblock \url{http://www.singular.uni-kl.de} (2016).

\bibitem{Goldman} W. M. Goldman. Topological components of spaces of representations. \emph{Invent. Math.} \textbf{93} (1988), 557-607.

\bibitem{Okay} C. Okay. Spherical posets from commuting elements. \emph{J. Group Theory} {\bf 21} (2018), 593-628.


\bibitem{OkayW} C. Okay and B. Williams. On the mod-$l$ homology of the classifying space for commutativity. Preprint (2018).  \href{https://arxiv.org/abs/1812.00142}{\texttt{{https://arxiv.org/abs/1812.00142}}} 
  
\bibitem{Torres} E. Torres-Giese and D. Sjerve. Fundamental groups of commuting elements in Lie groups.  \emph{Bull. London Math. Soc.} (2008) 40 (1), 65-76.  

\bibitem{Villarreal} B. Villarreal. Cosimplicial groups
  and spaces of homomorphisms. \emph{Algebr. Geom. Topol.} \textbf{17-6} (2017), 3519-3545.

\end{thebibliography}
\end{document}